\documentclass{jps-cp}

\usepackage{amsmath,amssymb,amsthm}

\theoremstyle{plain}
\newtheorem{theorem}{Theorem}[section]

\newtheorem{proposition}[theorem]{Proposition}
\newtheorem*{theorem*}{Theorem}
\newtheorem*{proposition*}{Proposition}
\newtheorem{corollary}[theorem]{Corollary}

\theoremstyle{definition}
\newtheorem{definition}[theorem]{Definition}
\newtheorem{ex}[theorem]{Example}
\theoremstyle{remark}
\newtheorem{remark}{Remark}[section]

\DeclareMathOperator{\id}{Id}

\usepackage{tikz}
\usetikzlibrary{topaths,calc}
\usetikzlibrary{positioning}
\tikzset{greynode/.style={circle,fill=gray!50,minimum size=0.4cm,inner sep=0pt},}
\tikzset{rednode/.style={circle,fill=black!100,minimum size=0.4cm,inner sep=0pt},}

\title{Interlacing Results for Hypergraphs}

\author{Raffaella \textsc{Mulas}$^{1,2}$}

\inst{$^{1}$The Alan Turing Institute, London NW1 2DB, UK \\
$^{2}$University of Southampton, Southampton SO17 1BJ, UK}

\recdate{March 31, 2021}

\abst{Hypergraphs are a generalization of graphs in which edges can connect any number of vertices. They allow the modeling of complex networks with higher-order interactions, and their spectral theory studies the qualitative properties that can be inferred from the spectrum, i.e.\ the multiset of the eigenvalues, of an operator associated to a hypergraph. It is expected that a small perturbation of a hypergraph, such as the removal of a few vertices or edges, does not lead to a major change of the eigenvalues. In particular, it is expected that the eigenvalues of the original hypergraph interlace the eigenvalues of the perturbed hypergraph. Here we work on hypergraphs where, in addition, each vertex--edge incidence is given a real number, and we prove interlacing results for the adjacency matrix, the Kirchoff Laplacian and the normalized Laplacian. Tightness of the inequalities is also shown. }

\kword{networks, hypergraphs, eigenvalues, spectral theory, interlacing}

\begin{document}
\maketitle
\section{Introduction}
Hypergraphs are a generalization of graphs in which edges can connect any number of vertices. They allow the modeling of bitcoin transactions \cite{bitcoin}, quantum entropies \cite{entropy}, chemical reaction networks \cite{JM2019}, cellular networks \cite{KlamtHausTheis}, social networks \cite{ZhangLiu}, neural networks \cite{Curto}, opinion formation \cite{LanchierNeufer}, epidemic networks \cite{BodoKatonaSimon}, transportation networks \cite{Andreotti}. Moreover, \emph{hypergraphs with real coefficients} have been introduced in \cite{JM2020} as a generalization of classical hypergraphs where, in addition, each vertex--edge incidence is given a real coefficient. These coefficients allow to model, for instance, the stoichiometric coefficients when considering chemical reaction networks, or the probability that a given vertex belongs to an edge. In \cite{JM2020}, also the adjacency matrix and the normalized Laplacian associated to such hypergraphs have been introduced, while the corresponding Kirchhoff Laplacian has been introduced in \cite{Yuji}.\newline
Here we study the spectral properties of these operators and we prove, in particular, interlacing results. We show that, given an operator $\mathcal{O}$ (which is either the adjacency matrix, the normalized Laplacian or the Kirchhoff Laplacian), then the eigenvalues of the operator $\mathcal{O}(G)$ associated to a hypergraph $G$ interlace the eigenvalues of $\mathcal{O}(G')$, if $G'$ is obtained from $G$ by deleting vertices or edges. We also prove the tightness of these inequalities.\newline
Since spectral theory studies the qualitative properties of a graph ---\,and, more generally, of a hypergraph\,--- that can be inferred by the spectra of its associated operators, interlacing results are meaningful as they offer a measure of how much a spectrum changes when deleting vertices or edges. We refer the reader to \cite{Horak,Butler,Reff,Haemers} for some literature on interlacing results in the case of graphs, simplicial complexes and hypergraphs.\newline
The paper is structured as follows. In Section \ref{section:Definitions} we offer an overview of the definitions on hypergraphs that will be needed throughout this paper. In Section \ref{section:MinMax} we recall the Courant--Fischer--Weyl min-max principle and we apply it to characterize the eigenvalues of the adjacency matrix, normalized Laplacian and Kirchhoff Laplacian associated to a hypergraph. 
In Section \ref{section:Cauchy} we apply the Cauchy interlacing Theorem to prove various interlacing results, and in Section \ref{Butler} we prove some additional interlacing results for the normalized Laplacian, using a generalization of the proof method developed by Butler in \cite{Butler}. Finally, in Section \ref{section:last} we draw some conclusions.

\section{Definitions}\label{section:Definitions}
We recall the basic definitions on hypergraphs with real coefficients, following \cite{JM2020}.
	\begin{definition}
			A \emph{hypergraph with real coefficients} (Fig. \ref{fig:hyp}) is a triple $G=(V,E,\mathcal{C})$ such that:
			\begin{itemize}
				\item $V=\{v_1,\ldots,v_n\}$ is a finite set of \emph{nodes} or \emph{vertices};
				\item $E=\{e_1,\ldots,e_m\}$ is a multiset of elements $e_j\subseteq V$ called \emph{edges};
				\item $\mathcal{C}=\{C_{v,e}:v\in V\text{ and }e\in E \}$ is a set of \emph{coefficients} $C_{v,e}\in\mathbb{R}$ and it is such that
				\begin{equation}\label{eq:zerocoeff}
				    C_{v,e}=0 \iff v\notin e.
				\end{equation}
			\end{itemize}
		\end{definition}
 \begin{figure}[tbh]
    \centering
\includegraphics[width=3.5cm]{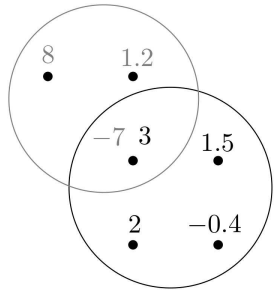}    \caption{A hypergraph with real coefficients that has $6$ vertices and $2$ edges.}
    \label{fig:hyp}
\end{figure}

From here on, we fix an hypergraph with real coefficients $G=(V,E,\mathcal{C})$ and we assume that each vertex is contained in at least one edge, that is, there are no isolated vertices.
\begin{definition}
    Given $e\in E$, its \emph{cardinality}, denoted $|e|$, is the number of vertices that are contained in $e$.
\end{definition}

\begin{remark}
The \emph{oriented hypergraphs} introduced by Reff and Rusnak in \cite{ReffRusnak} are hypergraphs with real coefficients such that $C_{v,e}\in\{-1,0,1\}$ for each $v\in V$ and $e\in E$. \emph{Signed graphs} are oriented hypergraphs such that $|e|=2$ for each $e\in E$, and \emph{simple graphs} are signed graphs such that, for each $e\in E$, there exists a unique $v\in V$ and there exists a unique $w\in V$ satisfying
\begin{equation*}
   C_{v,e}=-C_{w,e}=1.
\end{equation*}Moreover, \emph{weighted hypergraphs} are hypergraphs with real coefficients such that, for each $e\in E$ and for each $v\in e$, $C_{v,e}=:\omega(e)$ does not depend on $v$.
\end{remark}

\begin{definition}
	Given $v\in V$, its \emph{degree} is
	\begin{equation}\label{eq:defdegree}
	\deg (v):=\sum_{e\in E}(C_{v,e})^2.
	\end{equation}
The $n\times n$ diagonal \emph{degree matrix} of $G$ is
\begin{equation*}
    D:=D(G)=\textrm{diag}\bigl(\deg (v_i)\bigr)_{i=1,\ldots,n}.
\end{equation*}
\end{definition}
Note that $D$ is invertible, since we are assuming that there are no isolated vertices.

		\begin{definition}
The $n\times n$ \emph{adjacency matrix} of $G$ is $A:=A(G)=(A_{ij})_{ij}$, where $A_{ii}:=0$ for all $i=1,\ldots,n$ and
    \begin{equation*}
        A_{ij}:=-\sum_{e\in E}C_{v_i,e}\cdot C_{v_j,e}\quad \text{for all }i\neq j.
    \end{equation*}
\end{definition}
		\begin{definition}
The $n\times m$ \emph{incidence matrix} of $G$ is $\mathcal{I}:=\mathcal{I}(G)=(\mathcal{I}_{ij})_{ij}$, where
	\begin{equation*} 
	\mathcal{I}_{ij}:=C_{v_i,e_j}.
	\end{equation*}
\end{definition}

\begin{definition}
The \emph{normalized Laplacian} of $G$ is the $n\times n$ matrix
\begin{equation*}
    L:=L(G)=\id-D(G)^{-1/2}A(G)D(G)^{-1/2},
\end{equation*}
where $\id$ is the $n\times n$ identity matrix.
\end{definition}

\begin{remark}
In \cite{JM2020}, the normalized Laplacian is defined as the $n\times n$ matrix
\begin{equation*}
    \mathcal{L}(G):=\id-D(G)^{-1}A(G),
\end{equation*}which is not necessarily symmetric. Here we chose to work on $L(G)$, which generalizes the classical normalized Laplacian for graphs introduced by Fan Chung in \cite{Chung}, so that we can apply the properties of symmetric matrices. From a spectral point of view, working on $L(G)$ or $\mathcal{L}(G)$ is equivalent. In fact,
\begin{equation*}
    \mathcal{L}(G)=D(G)^{-1/2}L(G)D(G)^{1/2},
\end{equation*}hence the matrices $L(G)$ and $\mathcal{L}(G)$ are similar and, therefore, isospectral.
\end{remark}

The Kirchhoff Laplacian, in the context of hypergraphs with real coefficients, was introduced by Hirono et al.\ \cite{Yuji}. We recall it and we introduce the \emph{dual Kirchhoff Laplacian}.
\begin{definition}
The \emph{Kirchhoff Laplacian} of $G$ is the $n\times n$ matrix
\begin{equation*}
    K:=K(G)=\mathcal{I}(G)\cdot \mathcal{I}(G)^\top=D(G)-A(G).
\end{equation*}
The \emph{dual Kirchhoff Laplacian} of $G$ is the $m\times m$ matrix
\begin{equation*}
    K^*:=K^*(G):=\mathcal{I}(G)^\top \cdot \mathcal{I}(G).
\end{equation*}
\end{definition}
\begin{remark}\label{rmk:K*}
$K(G)$ and $K^*(G)$ have the same non-zero eigenvalues. It follows from the fact that, if $f$ and $g$ are linear operators, then the non-zero eigenvalues of $fg$ and $gf$ are the same.
\end{remark}

Given an $n\times n$ real symmetric matrix $Q$, its \emph{spectrum} consists of $n$ real eigenvalues, counted with multiplicity. We denote them by
\begin{equation*}
    \lambda_1(Q)\leq \ldots \leq \lambda_n(Q).
\end{equation*}
Since the \emph{trace} of an $n\times n$ matrix (i.e.\ the sum of its diagonal elements) equals the sum of its eigenvalues, we have
\begin{equation*}
 \sum_{i=1}^n\lambda_i(A)=0,\quad   \sum_{i=1}^n\lambda_i(L)=n,\quad\mbox{and}\quad \sum_{i=1}^n\lambda_i(K)=\sum_{i=1}^n \deg (v_i).
\end{equation*}
The idea of the interlacing results that we will prove is to show how the removal of part of the hypergraph effects the eigenvalues of its associated operators. We define two operations that can be done for removing part of a hypergraph: the \emph{vertex deletion} and the \emph{edge deletion}, as generalizations of the ones in \cite{Reff}. Similarly, we also define the \emph{restriction} of a hypergraph to a subset of edges.
\begin{definition}
Given $v\in V$, we let $G\setminus v:=(V\setminus \{v\}, E_v,\mathcal{C}_v)$, where:
\begin{itemize}
    \item $E_v:=\{e\setminus \{v\}: e\in E\}$;
    \item $\mathcal{C}_v:=\mathcal{C}\setminus\{C_{v,e}:e\in E\}$.
\end{itemize}
We say that $G\setminus v$ is obtained from $G$ by a \emph{vertex deletion} of $v$.
\end{definition}
\begin{definition}
Given $e\in E$, we let $G\setminus e:=(V, E\setminus \{e\},\mathcal{C}_e)$, where
\begin{equation*}
    \mathcal{C}_e:=\mathcal{C}\setminus\{C_{v,e}:v\in V\}.
\end{equation*}
We say that $G\setminus e$ is obtained from $G$ by an \emph{edge deletion} of $e$. More generally, given $F\subseteq E$, we denote by $G\setminus F$ the hypergraph obtained from $G$ by deleting all edges in $F$.
\end{definition}
\begin{definition}
Given $F\subseteq E$, the \emph{restriction} of $G$ to $F$ is $\left.G\right|_F:=(V_F,F,\left.\mathcal{C}\right|_F)$, where
\begin{itemize}
    \item $V_F:=\{v\in V: v\in e \text{ for some }e\in F\}$ and
    \item $\left.\mathcal{C}\right|_F:=\{C_{v,e}\in \mathcal{C}:v\in V_F \text{ and }e\in F\}$.
\end{itemize}
\end{definition}

\section{Min-max principle}\label{section:MinMax}
We recall the Courant--Fischer--Weyl min-max principle (Theorem 2.1 in \cite{Butler}):
\begin{theorem}[Min-max principle] Let $Q$ be an $n\times n$ real symmetric matrix. Let $\mathcal{X}^k$ denote a $k$--dimensional subspace of $\mathbb{R}^n$ and $\mathbf{x}\bot \mathcal{X}^k$ signify that $\mathbf{x}\bot \mathbf{y}$ for all $\mathbf{y}\in\mathcal{X}^k$. Then
\begin{equation*}
    \lambda_k(Q)=\min_{\mathcal{X}^{n-k-1}}\left(\max_{\mathbf{x}\bot \mathcal{X}^{n-k-1},\, \mathbf{x}\neq 0}\frac{\mathbf{x}^\top Q \mathbf{x}}{\mathbf{x}^\top \mathbf{x}}\right)=\max_{\mathcal{X}^k}\left(\min_{\mathbf{x}\bot \mathcal{X}^k,\,\mathbf{x}\neq 0}\frac{\mathbf{x}^\top Q \mathbf{x}}{\mathbf{x}^\top \mathbf{x}}\right)
\end{equation*}for $k=1,\ldots,n$.
\end{theorem}
In the case of the normalized Laplacian, by considering the substitution $\mathbf{x}=D^{1/2}\mathbf{y}$,
\begin{equation*}
    \frac{\mathbf{x}^\top L \mathbf{x}}{\mathbf{x}^\top \mathbf{x}}= \frac{(D^{1/2}\mathbf{y})^\top L (D^{1/2}\mathbf{y})}{(D^{1/2}\mathbf{y})^\top (D^{1/2}\mathbf{y})}=\frac{\mathbf{y}^\top (D^{1/2}LD^{1/2}) \mathbf{y}}{\mathbf{y}^\top D \mathbf{y}},
\end{equation*}where
\begin{equation*}
    \mathbf{y}^\top (D^{1/2}LD^{1/2}) \mathbf{y}=\mathbf{y}^\top K \mathbf{y}=\mathbf{y}^\top (\mathcal{I}\mathcal{I}^\top)\mathbf{y}=(\mathbf{y}^\top\mathcal{I})(\mathbf{y}^\top\mathcal{I})^\top=\sum_{e\in E}\left(\sum_{v_i\in V}y_i\cdot C_{v_i,e}\right)^2.
\end{equation*}Therefore, by the min-max principle,
\begin{equation}\label{eq:minmaxL}
    \lambda_k(L)=\min_{\mathcal{X}^{n-k-1}}\left(\max_{\mathbf{y}\bot \mathcal{X}^{n-k-1},\, \mathbf{y}\neq 0} \frac{\sum_{e\in E}\biggl(\sum_{v_i\in V}y_i\cdot C_{v_i,e}\biggr)^2}{\sum_{v_i\in V} y_i^2\deg(v_i)}\right).
\end{equation}Similarly,
\begin{equation}\label{eq:minmaxK}
    \lambda_k(K)=\min_{\mathcal{X}^{n-k-1}}\left(\max_{\mathbf{y}\bot \mathcal{X}^{n-k-1},\, \mathbf{y}\neq 0} \frac{\sum_{e\in E}\biggl(\sum_{v_i\in V}y_i\cdot C_{v_i,e}\biggr)^2}{\sum_{v_i\in V} y_i^2}\right)
\end{equation}and 
\begin{equation}\label{eq:minmaxA}
    \lambda_k(A)=\min_{\mathcal{X}^{n-k-1}}\left(\max_{\mathbf{y}\bot \mathcal{X}^{n-k-1},\, \mathbf{y}\neq 0} \frac{\sum_{e\in E}\sum_{v_i,v_j\in V,\,i\neq j}\biggl(-y_i\cdot y_j\cdot C_{v_i,e}\cdot C_{v_j,e}\biggr)}{\sum_{v_i\in V} y_i^2}\right).
\end{equation}
\begin{remark}
By the above characterizations, it is clear that the eigenvalues of $L$ and $K$ are non-negative.
\end{remark}

\section{Cauchy interlacing}\label{section:Cauchy}

We recall the Cauchy interlacing Theorem (Theorem 4.3.17 in \cite{matrix}) and we apply it in order to prove interlacing results for $A$, $L$ and $K$ when vertices or edges are removed. 

\begin{theorem}[Cauchy interlacing Theorem] Let $Q$ be an $n\times n$ real symmetric matrix and let $P$ be an $(n-1)\times (n-1)$ principal sub-matrix of $Q$. Then
\begin{equation*}
        \lambda_{k}(Q)\leq \lambda_k(P)\leq \lambda_{k+1}(Q)\quad\text{for all }k\in\{1,\ldots,n-1\}.
    \end{equation*}
\end{theorem}

\begin{corollary}Given $v\in V$,
\begin{enumerate}
    \item $\lambda_{k}(A(G))\leq \lambda_k(A(G\setminus v))\leq \lambda_{k+1}(A(G))$, for all $k\in\{1,\ldots,n-1\}$;
    \item $\lambda_{k}(K(G))\leq \lambda_k(K(G\setminus v))\leq \lambda_{k+1}(K(G))$, for all $k\in\{1,\ldots,n-1\}$;
    \item $\lambda_{k}(L(G))\leq \lambda_k(L(G\setminus v))\leq \lambda_{k+1}(L(G))$, for all $k\in\{1,\ldots,n-1\}$.
\end{enumerate}Given $e\in E$,
\begin{itemize}
    \item[4.] $\lambda_{k}(K(G))\leq \lambda_k(K(G\setminus e))\leq \lambda_{k+1}(K(G))$, for all $k\in\{1,\ldots,n-1\}$.
\end{itemize}
\end{corollary}

\begin{proof}
Since $A(G\setminus v)$ is an $(n-1)\times (n-1)$ principal sub-matrix of $A(G)$, the first claim follows from the Cauchy interlacing Theorem. The second and the third claim are analogous. Similarly, since $K^*(G\setminus e)$ is an $(m-1)\times (m-1)$ principal sub-matrix of $K^*(G)$, by the Cauchy interlacing Lemma we have that
\begin{equation*}
     \lambda_{t}(K^*(G))\leq \lambda_t(K^*(G\setminus e))\leq \lambda_{t+1}(K^*(G))\quad\text{for all }t\in\{1,\ldots,m-1\}.
\end{equation*}Since $K$ and $K^*$ have the same non-zero eigenvalues for any hypergraph (cf.\ Remark \ref{rmk:K*}), the claim follows.
\end{proof}

We now apply the Cauchy interlacing Theorem in order to prove the following

\begin{theorem}\label{thm:QM}
Let $Q$ be an $n\times n$ real symmetric matrix, let $M$ be an $m\times m$ real symmetric matrix and assume that there exists a principal sub-matrix of both $Q$ and $M$ of size $n-r=m-l$. Then,
\begin{equation*}
    \lambda_{k-l}(Q)\leq \lambda_k(M)\leq \lambda_{k+r}(Q), \quad\text{for all }k\in\{l+1,\ldots,n-r\}.
    \end{equation*}
\end{theorem}
\begin{proof}Let $P$ be a principal sub-matrix of both $Q$ and $M$, of size $n-r=m-l$. By repeatedly applying the Cauchy interlacing Theorem,
\begin{equation*}
\lambda_{j}(Q)\leq \lambda_j(P)\leq \lambda_{j+r}(Q)\quad\text{for all }j\in\{1,\ldots,n-r\}
\end{equation*}and
\begin{equation*}
\lambda_{j}(M)\leq \lambda_j(P)\leq \lambda_{j+l}(M)\quad\text{for all }j\in\{1,\ldots,m-l\}.
\end{equation*}Therefore,
\begin{equation*}
    \lambda_{j}(Q)\leq \lambda_j(P)\leq \lambda_{j+l}(M)\leq \lambda_{j+l}(P)\leq  \lambda_{j+l+r}(Q),
\end{equation*}for all $j\in\{1,\ldots,n-l-r\}$. Hence,
\begin{equation*}
    \lambda_{k-l}(Q)\leq \lambda_k(M)\leq \lambda_{k+r}(Q), \quad\text{for all }k\in\{l+1,\ldots,n-r\}.
\end{equation*}

\end{proof}

\begin{corollary}\label{cor:edges}Given $F\subseteq E$ such that $\left.G\right|_F$ has $t$ vertices,
\begin{equation}\label{eq:tA}
        \lambda_{k-t+1}(A(G))\leq \lambda_k(A(G\setminus F))\leq \lambda_{k+t-1}(A(G))
    \end{equation}for all $k\in\{t,\ldots,n-(t-1)\}$, and
    \begin{equation}\label{eq:tL}
        \lambda_{k-t}(L(G))\leq \lambda_k(L(G\setminus F))\leq \lambda_{k+t}(L(G)),
    \end{equation}for all $k\in\{t+1,\ldots,n-t\}$.
\end{corollary}
\begin{proof}
Let $w_1,\ldots,w_t$ be the vertices of $G|_F$. Then, $A(G\setminus\{w_1,\ldots,w_{t-1}\})$ is a principal sub-matrix of both $A(G)$ and $A(G\setminus F)$. Similarly, $L(G\setminus\{w_1,\ldots,w_{t}\})$ is a principal sub-matrix of both $L(G)$ and $L(G\setminus F)$. The claim follows by Theorem \ref{thm:QM}.
\end{proof}
\begin{remark}
Corollary \ref{cor:edges} is accurate in the sense that it is not possible to substitute, in the claim, the inequalities in \eqref{eq:tA} by 
    \begin{equation*}
        \lambda_{k-t+2}(A(G))\leq \lambda_k(A(G\setminus F))\leq \lambda_{k+t-2}(A(G)),
    \end{equation*}and the inequalities in \eqref{eq:tL} by
    \begin{equation*}
        \lambda_{k-t+1}(L(G))\leq \lambda_k(L(G\setminus F))\leq \lambda_{k+t-1}(L(G)).
    \end{equation*}The accuracy for $A$ can be easily seen by considering the case where $F$ consists of a single \emph{loop} $\ell$ (that is, an edge containing one single vertex). Clearly, removing one loop from the hypergraph does not change its adjacency matrix and therefore the inequalities in \eqref{eq:tA}, in this case, can be re-written as
    \begin{equation*}
        \lambda_{k}(A(G))\leq \lambda_k(A(G))\leq \lambda_{k}(A(G)).
    \end{equation*}The accuracy of \eqref{eq:tL} is shown by the next example.
    \end{remark}

    \begin{ex}\label{ex:loop}
    Let $G:=(V,E,\mathcal{C})$ be such that (Fig. \ref{fig:loop}):
    \begin{itemize}
        \item $V=\{v_1,v_2,v_3\}$;
        \item $E=\{e_1,\ell\}$, where $e_1=\{v_1,v_2,v_3\}$ and $\ell=\{v_1\}$;
        \item $\mathcal{C}_{v,e}=1$ for each $e\in E$ and each $v\in e$.
    \end{itemize}
      \begin{figure}[tbh]
    \centering
\includegraphics[width=8cm]{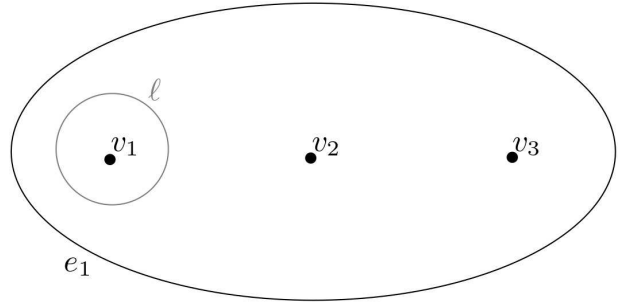}    \caption{The hypergraph in Example \ref{ex:loop}.}
    \label{fig:loop}
\end{figure}
    Then, 
    \begin{equation*}
        D(G)=\begin{pmatrix}
            2 & 0 & 0\\
           0 & 1 & 0\\
            0 & 0 & 1\\
        \end{pmatrix}, \quad  A(G)=\begin{pmatrix}
            0 & -1 & -1\\
           - 1 & 0 & -1\\
            -1 & -1 & 0\\
        \end{pmatrix}
            \end{equation*}
  and therefore
        \begin{equation*}
 L(G)=\begin{pmatrix}
            1 & 1/\sqrt{2} &  1/\sqrt{2}\\
             1/\sqrt{2} & 1 & 1\\
             1/\sqrt{2} & 1 & 1\\
        \end{pmatrix}.
    \end{equation*}Hence,
    \begin{equation*}
        \lambda_1(L(G))=0,\quad\lambda_2(L(G))=\frac{3-\sqrt{5}}{2},\quad \lambda_3(L(G))=\frac{3+\sqrt{5}}{2};
    \end{equation*}while
     \begin{equation*}
 L(G\setminus \ell)=\begin{pmatrix}
            1 & 1 &  1\\
             1 & 1 & 1\\
             1 & 1 & 1\\
        \end{pmatrix},
    \end{equation*}therefore
    \begin{equation*}
        \lambda_1(L(G\setminus \ell))=0,\quad \lambda_2(L(G\setminus \ell))=0,\quad \lambda_3(L(G\setminus \ell))=3.
    \end{equation*}In particular,
    \begin{equation*}
        \lambda_3(L(G\setminus \ell))>\lambda_3(L(G))
    \end{equation*}and
    \begin{equation*}
         \lambda_2(L(G\setminus \ell))<\lambda_2(L(G)).
    \end{equation*}This shows the accuracy of Corollary \ref{cor:edges} for $L$.
    \end{ex}
\section{Alternative interlacing for the normalized Laplacian}\label{Butler}
In the case when we remove an edge that is not a loop from a hypergraph, we can improve Corollary \ref{cor:edges} by partly generalizing, for $L$, Theorem 1.2 in \cite{Butler}.
\begin{theorem}\label{thm:Butler}
Given $\hat{e}\in E$ of cardinality $t\geq 2$,
    \begin{equation}\label{eq:hat}
        \lambda_{k-t+1}(L(G))\leq \lambda_k(L(G\setminus \hat{e})),
    \end{equation}for all $k\in\{t,\ldots,n\}$. More generally, given $F\subseteq E$ such that $|e|\geq 2$ for each $e\in F$ and $\sum_{e\in F}|e|=t$,
    \begin{equation*}
        \lambda_{k-t+|F|}(L(G))\leq \lambda_k(L(G\setminus F)).
    \end{equation*}
\end{theorem}

\begin{proof}
 Up to re-labeling of the vertices, assume that $\hat{e}=\{v_1,\ldots,v_t\}$ and let 
\begin{equation*}
    \mathcal{Z}:=\biggl\{\mathbf{e}_2-\mathbf{e}_1\cdot \left(\frac{C_{v_1,\hat{e}}}{(t-1)C_{v_2,\hat{e}}}\right),\ldots, \mathbf{e}_t-\mathbf{e}_1\cdot \left(\frac{C_{v_1,\hat{e}}}{(t-1)C_{v_t,\hat{e}}}\right) \biggr\}
\end{equation*}where $\mathbf{e}_1,\ldots,\mathbf{e}_t$ are the first $t$ standard unit vectors in $\mathbb{R}^n$ and therefore the condition $\mathbf{y}\bot \mathcal{Z}$ implies that
\begin{equation*}
    \sum_{v_i\in V}y_i\cdot C_{v_i,\hat{e}}=y_1\cdot C_{v_1,\hat{e}}+y_2\cdot C_{v_2,\hat{e}}+\ldots+y_t\cdot C_{v_t,\hat{e}}=y_1\cdot C_{v_1,\hat{e}}-(t-1)\cdot y_1\cdot \frac{C_{v_1,\hat{e}}}{(t-1)}=0.
\end{equation*}By \eqref{eq:minmaxL}, we have
\begin{align*}
    \lambda_k(L(G\setminus \hat{e}))&=\min_{\mathcal{X}^{n-k-1}}\left(\max_{\mathbf{y}\bot \mathcal{X}^{n-k-1},\, \mathbf{y}\neq 0} \frac{\sum_{e\in E\setminus \hat{e}}\biggl(\sum_{i\in V}y_i\cdot C_{v_i,e}\biggr)^2}{\sum_{v_i\in V} y_i^2\deg(v_i)-\sum_{v_i\in \hat{e}}y_i^2\cdot C_{v_i,\hat{e}}^2}\right)\\
    &=\min_{\mathcal{X}^{n-k-1}}\left(\max_{\mathbf{y}\bot \mathcal{X}^{n-k-1},\, \mathbf{y}\neq 0} \frac{\sum_{e\in E}\biggl(\sum_{v_i\in V}y_i\cdot C_{v_i,e}\biggr)^2-\biggl(\sum_{v_i\in V}y_i\cdot C_{v_i,\hat{e}}\biggr)^2}{\sum_{v_i\in V} y_i^2\deg(v_i)-\sum_{v_i\in \hat{e}}y_i^2\cdot C_{v_i\hat{e}}^2}\right)\\
    &\geq \min_{\mathcal{X}^{n-k-1}}\left(\max_{\mathbf{y}\bot \mathcal{X}^{n-k-1},\,\mathbf{y}\bot\mathcal{Z},\, \mathbf{y}\neq 0} \frac{\sum_{e\in E}\biggl(\sum_{v_i\in V}y_i\cdot C_{v_i,e}\biggr)^2}{\sum_{v_i\in V} y_i^2\deg(v_i)-\sum_{v_i\in \hat{e}}y_i^2\cdot C_{v_i,\hat{e}}^2}\right)\\
    &\geq \min_{\mathcal{X}^{n-k-1}}\left(\max_{\mathbf{y}\bot \mathcal{X}^{n-k-1},\,\mathbf{y}\bot\mathcal{Z},\, \mathbf{y}\neq 0} \frac{\sum_{e\in E}\biggl(\sum_{v_i\in V}y_i\cdot C_{v_i,e}\biggr)^2}{\sum_{v_i\in V} y_i^2\deg(v_i)}\right)\\
    &\geq \min_{\mathcal{X}^{n-k+t-2}}\left(\max_{\mathbf{y}\bot \mathcal{X}^{n-k+t-2},\, \mathbf{y}\neq 0} \frac{\sum_{e\in E}\biggl(\sum_{v_i\in V}y_i\cdot C_{v_i,e}\biggr)^2}{\sum_{v_i\in V} y_i^2\deg(v_i)}\right)\\
    &=\lambda_{k-t+1}(L(G)).
\end{align*}In the third line we added the condition $\mathbf{y}\bot\mathcal{Z}$, that makes the second term of the numerator vanish and that restricts the maximum over a smaller set. In the fifth line we considered an optimization that includes the one in the fourth line as particular case. This proves the claim.
\end{proof}

As shown by the next example, Theorem \ref{thm:Butler} is accurate in the sense that it is not possible to substitute, in the claim, the inequality \eqref{eq:hat} by 
    \begin{equation*}
        \lambda_{k-t+2}(L(G))\leq \lambda_k(L(G\setminus \hat{e}))
    \end{equation*}which becomes, for $|\hat{e}|=t=2$,
    \begin{equation*}
        \lambda_{k}(L(G))\leq \lambda_k(L(G\setminus \hat{e})).
    \end{equation*}

    \begin{ex}\label{ex:Butler}
    
    Let $G$ be the simple graph on $7$ nodes in Fig. \ref{fig:JMM}. By Theorem 2.1 and Theorem 3.1 in \cite{JMM}, 
\begin{equation*}
    \lambda_7(L(G))>\frac{4}{3}=\lambda_7(L(G\setminus\hat{e})).
\end{equation*}This shows the accuracy of Theorem \ref{thm:Butler}.
    \end{ex}
    
    \begin{figure}[tbh]
    \centering
\includegraphics[width=6cm]{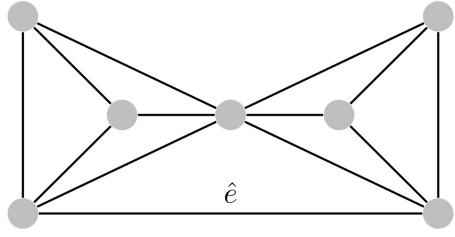}    \caption{The simple graph in Example \ref{ex:Butler}.}
    \label{fig:JMM}
\end{figure}


Now, while Theorem \ref{thm:Butler} considers the removal of edges of cardinality $\geq 2$, the following result is about the removal of loops.
\begin{proposition}If $\ell\in E$ is a loop, then
\begin{equation*}
    \lambda_k(G\setminus \ell)\geq \lambda_k (G)
\end{equation*}for all $k$ such that $\lambda_k (G)\geq 1$, and
\begin{equation*}
    \lambda_k(G\setminus \ell)\leq \lambda_k (G)
\end{equation*}for all $k$ such that $\lambda_k (G)\leq 1$.
\end{proposition}

\begin{proof}Assume that $\ell=\{v_1\}$ and $\lambda_k(G)\geq 1$. By \eqref{eq:minmaxL}, we have
\begin{align*}
    \lambda_k(L(G\setminus \ell))&=\min_{\mathcal{X}^{n-k-1}}\left(\max_{\mathbf{y}\bot \mathcal{X}^{n-k-1},\, \mathbf{y}\neq 0} \frac{\sum_{e\in E}\biggl(\sum_{v_i\in V}y_i\cdot C_{v_i,e}\biggr)^2-y_1^2\cdot C_{v_1,\ell}^2}{\sum_{v_i\in V} y_i^2\deg(v_i)-y_1^2\cdot C_{v_1,\ell}^2}\right)\\
    &\geq \min_{\mathcal{X}^{n-k-1}}\left(\max_{\mathbf{y}\bot \mathcal{X}^{n-k-1},\, \mathbf{y}\neq 0} \frac{\sum_{e\in E}\biggl(\sum_{v_i\in V}y_i\cdot C_{v_i,e}\biggr)^2}{\sum_{v_i\in V} y_i^2\deg(v_i)}\right)\\
    &=\lambda_{k}(L(G)),
\end{align*}since adding the same non-negative quantity to the numerator and to the denominator makes the resultant fraction closer to $1$. This proves the first claim. The proof of the second claim is analogous.
\end{proof}

\section{Conclusions}\label{section:last}
We have shown that, if the structure of a hypergraph with real coefficients is perturbed by removing (or adding) vertices and edges, then the eigenvalues of the perturbed hypergraph interlace those of the original hypergraph. We have proved various inequalities for each operator that we considered (adjacency matrix, Kirchhoff Laplacian and normalized Laplacian), and we have shown tightness of the inequalities. These results are in line with intuition, because the spectra of the operators associated to a hypergraph encode important structural properties of the hypergraph. For future directions, it will be interesting to apply these interlacing results to problems arising in both pure mathematics and applied network analysis.\newline

\textbf{Acknowledgments.} \newline
The author would like to thank Tetsuo Hatsuda (RIKEN iTHEMS) and the organizers of the International Conference on Blockchains and their Applications (Kyoto 2021) for the invitation. This work was supported by The Alan Turing Institute under the EPSRC grant EP/N510129/1.

\end{document}